\documentclass[12pt,reqno]{amsart}
\usepackage{amssymb,latexsym}
\topskip  \allowdisplaybreaks

\numberwithin{equation}{section}
\newtheorem{theorem}{Theorem}[section]

\newtheorem{corollary}[theorem]{Corollary}

\newtheorem{lemma}[theorem]{Lemma}

\begin{document}

\pagenumbering{arabic}
\pagestyle{headings}
\def\sof{\hfill\rule{2mm}{2mm}}
\def\ls{\leq}
\def\gs{\geq}
\def\SS{\mathcal S}
\def\qq{{\bold q}}
\def\txx{{\frac1{2\sqrt{x}}}}

\title{Smooth words and Chebyshev polynomials}
\author[A. Knopfmacher]{Arnold Knopfmacher}

\address{The John Knopfmacher Centre for Applicable Analysis and Number Theory, School of Mathematics, University of the Witwatersrand,
Johannesburg, South Africa.} \email{Arnold.Knopfmacher@ wits.ac.za}
\author[T. Mansour]{Toufik Mansour}
\address{Department of Mathematics, University of Haifa, 31905 Haifa,
Israel.} \email{toufik@math.haifa.ac.il}

\author[A. Munagi]{Augustine Munagi}
\address{The John Knopfmacher Centre for Applicable Analysis and Number Theory, School of Mathematics, University of the Witwatersrand,
Johannesburg, South Africa.} \email{Augustine.Munagi@wits.ac.za}
\author[H.~Prodinger]{Helmut Prodinger}
\address{Department of Mathematics\\
University of Stellenbosch\\
7602 Stellenbosch\\
South Africa}
\email{hproding@sun.ac.za}

\maketitle

%===========================================================================
\begin{abstract}
A word $\sigma=\sigma_1\cdots\sigma_n$ over the  alphabet $[k]=\{1,2,\ldots,k\}$ is said to be {\em smooth} if there are no two adjacent letters with difference greater than $1$.
A word $\sigma$ is said to be {\em smooth cyclic} if it is a smooth word and in addition satisfies $|\sigma_n-\sigma_1|\le 1$.
We find the explicit generating functions for the number of smooth words and cyclic smooth words in  $[k]^n$, in terms of {\it Chebyshev polynomials of the second kind}.
Additionally, we find explicit formula for the numbers themselves, as trigonometric sums. These lead to  immediate asymptotic corollaries. We also enumerate smooth necklaces, which are cyclic smooth words that are not equivalent
up to rotation.
\end{abstract}

2000 Mathematics Subject Classification: 68R05, 05A05, 05A15, 05A16\\

Keywords: smooth word, smooth cyclic word, smooth necklace, Chebyshev polynomial, generating function, tridiagonal matrix.
%===========================================================================
\section{Introduction}
Let $[k]^n$ be the set of all words of length $n$ over the alphabet $[k]=\{1,2,\ldots,k\}$.
In the past decade, many research papers have been devoted to the study of enumeration problems on the set $[k]^n$. For instance, the enumeration of words which contain a prescribed number of a given set of strings as substrings is a classical problem in combinatorics. This problem can, for example, be attacked using the transfer matrix method, see \cite[Section 4.7]{St1} and \cite{GJ}.
R\'egnier and Szpankowski \cite{RS} used a combinatorial approach to study the frequency of occurrences of certain
strings (which they also call a ``pattern") in a random word, where overlapping copies of the strings are counted separately.
Burstein and Mansour \cite{BM} have considered the enumeration of elements of $[k]^n$ that satisfy certain restrictions, where the restrictions may be characterized in terms of pattern avoidance.

In this paper we consider the enumeration of a special class of words, namely the {\it smooth words}. Intuitively the name ``smooth" indicates that one can move between adjacent letters of a word with at most a very slight ``bump."
 The precise definition follows.

A word $\sigma\in [k]^n$ is said to be {\em smooth} if it avoids a (contiguous) string of the form $ij$ where $|i-j|>1$.
%if $\sigma_{i+1}-\sigma_i\in\{0,1,-1\}$ for all $i=1,2,\ldots,n-1$.
For example, there are $7$ smooth words in $[3]^2$, namely $11$, $12$, $21$, $22$, $23$, $32$ and $33$.
 A word $\sigma$ is said to be {\em smooth cyclic} if it is a smooth word and in addition satisfies $\sigma_n-\sigma_1\in\{0,1,-1\}$. Clearly, each word in $[k]^n$, $k=1,2$, is smooth cyclic (and thus also smooth).
 We denote the number of smooth words (respectively, smooth cyclic words) in $[k]^n$ by $\textit{sw}_{n,k}$ (respectively, $\textit{scw}_{n,k}$). Table~\ref{tab1} shows the numbers of smooth words and smooth cyclic words of length
 $n$ over the alphabet $[k]$ for $0\le n\le 11$ and $3\le k\le 7$.

\begin{table}[htp]
\begin{tabular}{c|cccccccccccc}
  $n$        & 0 & 1 & 2 & 3 & 4 & 5 & 6 & 7  & 8  & 9  & 10  & 11     \\\hline\hline
  $\textit{sw}_{n,3}$& 1 & 3 & 7 & 17& 41& 99&239&577 &1393&3363&8119 &19601  \\
  $\textit{scw}_{n,3}$& 1 &3 & 7 & 15& 35& 83&199&479 &1155&2787&6727 &16239  \\\hline
  $\textit{sw}_{n,4}$& 1 & 4 &10 & 26& 68&178&466&1220&3194&8362&21892&57314\\
  $\textit{scw}_{n,4}$& 1 &4 &10 & 22& 54&134&340&872 &2254&5854&15250&39802\\\hline
  $\textit{sw}_{n,5}$& 1 & 5& 13& 35& 95& 259& 707& 1931& 5275& 14411& 39371& 107563\\
  $\textit{scw}_{n,5}$& 1 &5& 13& 29& 73& 185& 481& 1265& 3361& 8993 & 24193& 65345\\\hline
  $\textit{sw}_{n,6}$& 1 & 6& 16& 44& 122& 340& 950& 2658& 7442& 20844& 58392& 163594\\
  $\textit{scw}_{n,6}$& 1 &6& 16& 36& 92 & 236& 622& 1658& 4468& 12132& 33146& 90998\\\hline
  $\textit{sw}_{n,7}$& 1 & 7& 19& 53& 149& 421& 1193& 3387& 9627& 27383& 77923& 221805\\
  $\textit{scw}_{n,7}$& 1 & 7 & 19& 43& 111& 287& 763& 2051 & 5575& 15271& 42099& 116651\\\hline
\end{tabular}\vspace{7pt}
\caption{Numbers of smooth words and smooth cyclic words $\textit{sw}_{n,k}, \textit{scw}_{n,k}$.}\label{tab1}
\end{table}

We will find the explicit generating functions for $\textit{sw}_{n,k}$ and $\textit{scw}_{n,k}$ in terms of {\it Chebyshev polynomials of the second kind}.
Additionally, we find explicit formul{\ae} for the numbers themselves, in terms of certain trigonometric expressions. They allow for immediate asymptotic corollaries.

Chebyshev polynomials of the second kind are defined by
            $$U_r(\cos\theta)=\frac{\sin(r+1)\theta}{\sin\theta}$$
for $r\gs0$. Evidently, $U_r(x)$ is a polynomial of degree $r$ in $x$ with integer coefficients. For example, $U_0(x)=1$, $U_1(x)=2x$, $U_2(x)=4x^2-1$, and in general,
\begin{equation}\label{eqcheb}
U_r(x)=2xU_{r-1}(x)-U_{r-2}(x).
\end{equation}
{\em Chebyshev polynomials of the first kind} are defined by $T_r(\cos\theta)=\cos(r\theta)$ which is equivalent to
$T_r(x)=\frac{1}{2}(U_r(x)-U_{r-2}(x))$.  Chebyshev polynomials were invented for the needs of approximation theory, but are also widely used in various other branches of mathematics, including algebra, combinatorics, and number theory (see \cite{Ri}).

Two words $\sigma=\sigma_1\cdots\sigma_n$ and $\pi=\pi_1\cdots\pi_n$ in $[k]^n$ are said to be {\em rotation equivalent} if there exists an index $i,\ 1\leq i\leq n$ such that $\pi_i\pi_{i+1}\cdots\pi_n\pi_1\pi_2\cdots\pi_{i-1}=\sigma$.
 For example, the words $122$, $212$ and $221$ are rotation equivalent. The set of {\em necklaces} of length $n$ over the alphabet $[k]$ is the set of words in $[k]^n$ up to the rotation-equivalence.
For example, if $k=2$ and $n=3$ there are $4$ necklaces, namely, $111$, $122$ ($212$ and $221$ are rotation equivalent), $112$ ($121$ and $211$ are rotation equivalent) and $222$. Using our results on smooth cyclic words, we also determine the number of smooth necklaces in $[k]^n$.

The paper is organized as follows. In Section \ref{smooth} we obtain the generating function for the number $\textit{sw}_{n,k}$ of smooth words, followed shortly by the explicit enumeration formula (Theorems \ref{th1} and \ref{th11}).
The asymptotic growth rate is also obtained in this section. In Section \ref{smoothcyc} we obtain the corresponding enumeration results for smooth cyclic words. Lastly, Section \ref{smoothneck} deals with the enumeration of smooth necklaces.

\section{Enumeration of Smooth words}\label{smooth}
Let $\textit{sw}_k(x)$ denote the generating function for the number of smooth words over $[k]$:
\[\textit{sw}_k(x)=\sum_{n\ge 0} \textit{sw}_{n,k}x^n.\]
In order to obtain a formula for $\textit{sw}_k(x)$, we introduce the following notations. Let $\textit{sw}_k(x\mid i_1i_2\cdots i_s)$ be the generating function for the number of smooth words $\sigma_1\cdots\sigma_n$ of length $n$ over the alphabet $[k]$ such that $\sigma_1\cdots\sigma_s=i_1\cdots i_s$.

\begin{lemma}\label{lem1}
The generating function $\textit{sw}_k(x\mid i)$ satisfies
$$\textit{sw}_k(x\mid i)=x+x\bigl(\textit{sw}_k(x\mid i-1)+\textit{sw}_k(x\mid i)+\textit{sw}_k(x\mid i+1)\bigr),$$
for all $1\leq i\leq k$, where $\textit{sw}_k(x\mid i)=0$ if $i\not\in[k]$.
\end{lemma}
\begin{proof}
Let $\sigma$ be any nonempty smooth word. If $\sigma$ contains exactly one letter then $\sigma_1\in[k]$. Otherwise, the second letter of $\sigma=i\sigma_2\cdots\sigma_{n}$ is either $i-1,i$ or $i+1$. Thus, in terms of generating functions we have that
  $$\textit{sw}_k(x\mid i)=x+x\bigl(\textit{sw}_k(x\mid i-1)+\textit{sw}_k(x\mid i)+\textit{sw}_k(x\mid i+1)\bigr),$$
for all $1\leq i\leq k$, which completes the proof.
\end{proof}

Rewriting Lemma~\ref{lem1} as a matrix system we obtain
\begin{equation}\label{sys1}
\textbf{A}\begin{pmatrix}\textit{sw}_k(x\mid 1)\\\vdots\\\textit{sw}_k(x\mid k)\end{pmatrix} =\begin{pmatrix}1\\\vdots\\1\end{pmatrix}x,
\end{equation}
where
$\textbf{A}=(a_{ij})$ is a $k\times k$ matrix defined by $a_{ii}=1-x$, $a_{i(i+1)}=a_{(i+1)i}=-x$, and $a_{ij}=0$ for all $|i-j|>1$. Clearly, $\textbf{A}$ is a tridiagonal matrix.

Applying a result of Usmani~\cite{Usm}\footnote{Equivalent results have been published earlier, for instance in \cite{PaPr85}.} on the inversion of $\textbf{A}$ we get
$$(\textbf{A}^{-1})_{ij}=\begin{cases}
x^{j-i}\theta_{i-1}\theta_{k-j}/\theta_k& i\leq j,\\[4pt]
x^{i-j}\theta_{j-1}\theta_{k-i}/\theta_k& i>j,
\end{cases}$$
where $\theta_i$ satisfies the recurrence relation $\theta_i=(1-x)\theta_{i-1}-x^2\theta_{i-2}$, with the initial conditions $\theta_0=1$ and $\theta_1=1-x$. It follows from \eqref{eqcheb} that the solution is given by $\theta_i=x^iU_i\left(\frac{1-x}{2x}\right)$. Consequently
\begin{equation}\label{invA}
(\textbf{A}^{-1})_{ij}=
\begin{cases}
\dfrac{U_{i-1}\left(\frac{1-x}{2x}\right)U_{k-j}\left(\frac{1-x}{2x}\right)}
{xU_k\left(\frac{1-x}{2x}\right)}& i\leq j,\\[14pt]
\dfrac{U_{j-1}\left(\frac{1-x}{2x}\right)U_{k-i}\left(\frac{1-x}{2x}\right)}{xU_k\left(\frac{1-x}{2x}\right)}& i>j.
\end{cases}
\end{equation}
Thus the solution of \eqref{sys1} is
$$\begin{pmatrix}\textit{sw}_k(x\mid 1)\\\vdots\\\textit{sw}_k(x\mid k)\end{pmatrix} =\textbf{A}^{-1}\begin{pmatrix}1\\\vdots\\1\end{pmatrix}x.$$

This implies that the generating function $\textit{sw}_k(x\mid i)$ is given by
$$\textit{sw}_k(x\mid i)=\frac{1}{U_k(t)}\biggl[U_{k-i}(t)\sum_{j=0}^{i-2}U_j(t)+U_{i-1}(t)\sum_{j=0}^{k-i}U_j(t)\biggr],$$
where $t=\frac{1-x}{2x}$.

In order to simplify the right-hand side we use the identity
\begin{equation}\label{eqid1}
\sum_{j=0}^pU_j(t)=\frac{U_{p+1}(t)-U_{p}(t)-1}{2(t-1)},
\end{equation}
which may be proved easily from the fact that
$$\sum_{j=0}^{p}\sin(jt)=\frac{\sin((p+1)t)(\cos(t)-1)+\sin(t)\cos((p+1)t)-\sin(t)}{2(\cos(t)-1)}$$
and $T_n(x)=\frac{1}{2}(U_n(x)-U_{n-2}(x))$.%
\footnote{The verification of such identities is \emph{a priori} trivial and can be done by a computer,
since, upon rewriting the trigonometric functions via Euler's formul{\ae}, one only has to sum some finite geometric series.}

Thus
$$\textit{sw}_k(x\mid i)=\frac{1}{2(t-1)U_k(t)}\left[U_{i-1}(t)U_{k+1-i}(t)-U_{k-i}(t)U_{i-2}(t)-U_{k-i}(t)-U_{i-1}(t)\right].$$
Now apply the identity
\begin{equation}\label{eqid2}
U_i(t)U_j(t)=\frac{U_{i-j}(t)-tU_{i-j-1}(t)-U_{i+j+2}(t)+tU_{i+j+1}(t)}{2(1-t^2)},
\end{equation}
to obtain
\begin{equation*}
\textit{sw}_k(x\mid i)=
\frac{1}{2(t-1)U_k(t)}\left[\frac{U_{k}(t)-tU_{k-1}(t)-U_{k+2}(t)+tU_{k+1}(t)}{2(1-t^2)}-U_{k-i}(t)-U_{i-1}(t)\right],
\end{equation*}
which, by ~\eqref{eqcheb}, is equivalent to
\begin{equation}\label{eqiA}
\textit{sw}_k(x\mid i)=\frac{1}{2(t-1)U_k(t)}\left[U_k(t)-U_{k-i}(t)-U_{i-1}(t)\right].
\end{equation}

Now, using $\textit{sw}_k(x)=1+\sum_{i=1}^k\textit{sw}_k(x\mid i)$ and again ~\eqref{eqid1}, we obtain an explicit formula for the generating function $\textit{sw}_k(x)$.

\begin{theorem}\label{th1}
The generating function $\textit{sw}_k(x)$ for the number of smooth words of length $n$ over the alphabet $[k]$ is given by
$$\textit{sw}_k(x)=1+\frac{x(k-(3k+2)x)}{(1-3x)^2}+\frac{2x^2}{(1-3x)^2}\frac{1+U_{k-1}\left(\frac{1-x}{2x}\right)}{U_k\left(\frac{1-x}{2x}\right)}.$$
\end{theorem}

It is easy to see that each word in $[k]^n$, $k=1,2$, is a smooth word. For small values of $k$, Theorem \ref{th1} gives
\begin{itemize}
\item $\textit{sw}_3(x)=\frac{1+x}{1-2x-x^2}$, that is, the number of smooth words in $[3]^n$ is given by
$$\frac{1}{2}\bigl(1+\sqrt{2}\,\bigr)^{n+1}+\frac{1}{2}\bigl(1-\sqrt{2}\,\bigr)^{n+1}.$$

\item $\textit{sw}_4(x)=\frac{1+x-x^2}{1-3x+x^2}$, that is, the number of smooth words in $[4]^n$ is given by
$$\frac{2}{\sqrt{5}}\biggl(\frac{1+\sqrt{5}}{2}\,\biggr)^{2n+1}
-\frac{2}{\sqrt{5}}\biggl(\frac{1-\sqrt{5}}{2}\,\biggr)^{2n+1}=2F_{2n+1}.$$

\item $\textit{sw}_5(x)=\frac{1+2x-2x^2-2x^3}{(1-x)(1-2x-2x^2)}$, that is, the number of smooth words in $[5]^n$ is given by
$$\frac{2+\sqrt{3}}{6}\bigl(1+\sqrt{3}\,\bigr)^{n+1}
+\frac{2-\sqrt{3}}{6}\bigl(1-\sqrt{3}\,\bigr)^{n+1}
+\frac{1}{3}.$$
\end{itemize}

In order to obtain an explicit formula for the number of smooth words of length $n$ over the alphabet $[k]$ we need the following lemma.

\begin{lemma}\label{lem1u}
Let $m\geq1$. Then
$$\frac{1}{U_m(x)}=\frac{1}{m+1}\sum_{j=1}^m\frac{(-1)^{j+1}\sin^2\big(\frac{j\pi}{m+1}\big)}{x-\cos\big(\frac{j\pi}{m+1}\big)}$$
and
$$\frac{1+U_{m-1}(x)}{U_{m}(x)}=
\frac{1}{m+1}\sum_{j=1}^m\frac{(1+(-1)^{j+1})\sin^2\left(\frac{j\pi}{m+1}\right)}{x-\cos\big(\frac{j\pi}{m+1}\big)}.$$
\end{lemma}
\begin{proof}
Let us compute the partial fraction decomposition of $\frac{1}{U_m(x)}$.
By general principles, it is
$\sum_{j=1}^m\frac{A_{m,j}}{x-\rho_{m,j}}$,
where $\rho_{m,j}=\cos\bigl(\frac{j\pi}{m+1}\bigr)$ are the zeros of the $m$-th Chebyshev polynomials of the second kind. Now, $A_{m,j}=\frac1{U'_m(\rho_{m,j})}$ and note that
\begin{equation*}
U'_m(x)=\frac{dU_m(x)}{dx}=\frac d{d\theta}\Big(\frac{\sin(m+1)\theta}{\sin \theta}\Big)\cdot\frac{d\theta}{dx}.
\end{equation*}
We work out that
\begin{equation*}
\frac{dU_m}{d\theta}=\frac{(m+1)\cos(m+1)\theta\cdot \sin\theta-\sin(m+1)\theta\cdot \cos\theta}{\sin^2\theta},
\end{equation*}
and if we plug in $x=\rho_{m,j}$ simplification occurs, since certain terms are just zero; we obtain that
\begin{align*}
\frac{dU_m}{d\theta}(\arccos\rho_{m,j})&=\frac{(m+1)\cos(m+1)\theta\cdot \sin\theta-
\sin(m+1)\theta\cdot \cos\theta}{\sin^2\theta}\bigg|_{\theta=\frac{\pi j}{m+1}}\\
&=\frac{(m+1)\cos(\pi j)\cdot \sin\frac{\pi j}{m+1}-
\sin(\pi j)\cdot \cos\frac{\pi j}{m+1}}{\sin^2\frac{\pi j}{m+1}}\\
&=\frac{(m+1)\cos(\pi j)\cdot \sin\frac{\pi j}{m+1}}{\sin^2\frac{\pi j}{m+1}}\\
&=\frac{(m+1)(-1)^j}{\sin\frac{\pi j}{m+1}}.
\end{align*}
Further $\frac{dx}{d\theta}=-\sin\theta$, so together $\frac{1}{A_{m,j}}=U'_m(\rho_{m,j})=\frac{(m+1)(-1)^{j+1}}{\sin^2\frac{\pi j}{m+1}}$, which completes the proof of the first identity. Similarly, the second one can be obtained.
\end{proof}

Now we are ready to obtain an explicit formula for the number of smooth words.

\begin{theorem}\label{th11}
The number of smooth words of length $n$ over the alphabet $[k]$ is given by
\begin{equation*}
\textit{sw}_{n,k}=\frac{1}{k+1}\sum_{j=1}^k(1+(-1)^{j+1})\cot^2\frac{j\pi}{2(k+1)}
\Big(1+2\cos\frac{j\pi}{k+1}\Big)^{n-1},
\end{equation*}
or alternatively as
\begin{equation*}
\textit{sw}_{n,k}=\frac{2}{k+1}\sum_{0\le j\le\frac{k-1}{2}}\cot^2\frac{(2j+1)\pi}{2(k+1)}
\Big(1+2\cos\frac{(2j+1)\pi}{k+1}\Big)^{n-1}.
\end{equation*}

\end{theorem}
\begin{proof}
Fix $k$ and let $\theta_j=\frac{j\pi}{k+1}$. Lemma~\ref{lem1u} says that the coefficient of $x^n$ in $\frac{2x^2}{(1-3x)^2}\frac{1+U_{k-1}(t)}{U_{k}(t)}$ with $t=\frac{1-x}{2x}$ (see Theorem~\ref{th1}) is given by
\begin{align}
\begin{split}
p_{n,k}&=[x^n]\frac{2x^2}{(1-3x)^2}\frac1{k+1}\sum_{j=1}^k\frac{(1+(-1)^{j+1})\sin^2\theta_j}
{\frac{1-x}{2x}-\cos\theta_j}\\
\label{eqc11}
&=\frac4{k+1}[x^{n-3}]\frac{1}{(1-3x)^2}\sum_{j=1}^k\frac{(1+(-1)^{j+1})\sin^2\theta_j}
{1-x\left(1+2\cos\theta_j\right)}.
\end{split}
\end{align}
Now notice that
$$\frac{1}{(1-3x)^2(1-x\omega)}=-\frac{3}{(\omega-3)(1-3x)^2}-\frac{3\omega}{(\omega-3)^2(1-3x)}+
\frac{\omega^2}{(\omega-3)^2(1-x\omega)}$$
and
\begin{equation}\label{eqc22}
\omega-3=1+2\cos\theta_j-3=-4\sin^2(\frac{\theta_j}{2}).
\end{equation}
So we are dealing with
$$A=\frac{3}{4\sin^2(\frac{\theta_j}{2})(1-3x)^2}-\frac{3\left(1+2\cos\theta_j\right)}
{16\sin^4(\frac{\theta_j}{2})(1-3x)}+\frac{(1+2\cos\theta_j)^2}{16\sin^4(\frac{\theta_j}{2})(1-x(1+2\cos\theta_j))},$$
which implies that the coefficient of $x^{n-3}$ in $A$ is given by
$$\frac{3^{n-2}(n-2)}{4\sin^2(\frac{\theta_j}{2})}-\frac{3^{n-2}(1+2\cos\theta_j)}{16\sin^4(\frac{\theta_j}{2})}+
\frac{(1+2\cos\theta_j)^{n-1}}{16\sin^4(\frac{\theta_j}{2})}.$$
%Using the fact that
%$$1+2\cos\frac{j\pi}{k+1}=\frac{\sin\frac{3j\pi}{2(k+1)}}{\sin\frac{j\pi}{2(k+1)}}$$
Hence, \eqref{eqc11} can be written as
$$p_{n,k}=\frac{4}{k+1}\sum_{j=1}^k(1+(-1)^{j+1})\sin^2\theta_j
\bigg[\frac{3^{n-2}(n-2)}{4\sin^2(\frac{\theta_j}{2})}-\frac{3^{n-2}(1+2\cos\theta_j)}{16\sin^4(\frac{\theta_j}{2})}+
\frac{(1+2\cos\theta_j)^{n-1}}{16\sin^4(\frac{\theta_j}{2})}
\bigg]$$
which simplifies to
$$p_{n,k}=\frac1{k+1}\sum_{j=1}^k(1+(-1)^{j+1})\cos^2(\frac{\theta_j}{2})
\bigg[4(n-2)3^{n-2}-\frac{3^{n-2}(1+2\cos\theta_j)}{\sin^2(\frac{\theta_j}{2})}+
\frac{(1+2\cos\theta_j)^{n-1}}{\sin^2(\frac{\theta_j}{2})}
\bigg].$$
Using the identity
$\sum_{j=1}^k(1+(-1)^{j+1})\cos^2(\frac{\theta_j}{2})=\frac{k+1}{2}$,
we get that
$$p_{n,k}=2(n-2)3^{n-2}+\frac1{k+1}\sum_{j=1}^k(1+(-1)^{j+1})\cot^2(\frac{\theta_j}{2})
\Big[(1+2\cos\theta_j)^{n-1}-3^{n-2}(1+2\cos\theta_j)\Big].$$
Note that
\begin{align*}
[x^n]\frac{x(k-(3k+2)x)}{(1-3x)^2}&=
[x^{n-1}]\frac{k}{(1-3x)^2}-[x^{n-2}]\frac{3k+2}{(1-3x)^2}\\
&=kn3^{n-1}-(n-1)(3k+2)3^{n-2}\\
&=(3k-2n+2)3^{n-2}.
\end{align*}
Therefore, Theorem~\ref{th1} gives
\begin{align*}
\textit{sw}_{n,k}&=(3k-2n+2)3^{n-2}+p_{n,k}\\
&=(3k-2)3^{n-2}\\&\qquad+\frac1{k+1}\sum\limits_{j=1}^k(1+(-1)^{j+1})\cot^2(\frac{\theta_j}{2})
\Big[(1+2\cos\theta_j)^{n-1}-3^{n-2}(1+2\cos\theta_j)\Big]\\
&=(3k-2)3^{n-2}\\&\qquad+\frac1{k+1}\sum\limits_{j=1}^k(1+(-1)^{j+1})\cot^2(\frac{\theta_j}{2})(1+2\cos\theta_j)
\Big[(1+2\cos\theta_j)^{n-2}-3^{n-2}\Big].
\end{align*}
Notice further that
$$\sum_{j=1}^k(1+(-1)^{j+1})\cos^2(\frac{\theta_j}{2})(1+2\cos\theta_j)=(k+1)(3k-2),$$
thus we have that
$$\textit{sw}_{n,k}=\frac1{k+1}\sum_{j=1}^k(1+(-1)^{j+1})\cot^2(\frac{\theta_j}{2})
(1+2\cos\theta_j)^{n-1},$$
as claimed.
\end{proof}

\begin{corollary}
Asymptotically, we have as $n \to \infty$,
\begin{equation*}
\textit{sw}_{n,k}\sim\frac{2}{k+1}\cot^2\frac{\pi}{2(k+1)}
\Big(1+2\cos\frac{\pi}{k+1}\Big)^{n-1}.
\end{equation*}
\end{corollary}
Note that since there are $k$ possible initial letters for a smooth word and at most three possibilities thereafter for each subsequent letter, the number of smooth words is bounded above by $k3^{k-1}$.
This upper bound becomes increasingly more accurate as $k$ grows larger (except for a factor of $\frac{8}{\pi ^2}$ below), since as $k \to \infty$,
\[1+2\cos\frac{\pi}{k+1}=3-\frac{\pi ^2}{k^2}+\frac{2 \pi ^2}{k^3}+O(k^{-4})
\]
and
\[\frac{2}{k+1}\cot^2\frac{\pi}{2(k+1)}=\frac{8 k}{\pi ^2}+\frac{8}{\pi ^2}+O\bigl(\frac{1}{ k}\bigr).
\]

\section{Smooth cyclic words} \label{smoothcyc}
In this section we find an explicit formula for the generating function
$\textit{scw}_k(x)=\sum_{n\ge 0} \textit{scw}_{n,k}x^n$.

Denote by $\textit{scw}_k(x\mid i_1\cdots i_s\mid j)$ the generating function for the number of smooth cyclic words $\sigma=\sigma_1\cdots\sigma_n$ of length $n$ over $[k]$ such that $\sigma_1\cdots\sigma_s=i_1\cdots i_s$ and $\sigma_n=j$. We define
    $$\textit{scw}_k(x,v\mid i_1\cdots i_s)=\sum_{j=1}^k \textit{scw}_k(x\mid i_1\cdots i_s\mid j)v^j.$$

\begin{lemma}\label{lemscw1}
For all $i=1,2,\ldots,k$,
\begin{align*}
\textit{scw}_k(x,v\mid i)&=(v^{i-1}[\![i>1]\!]+v^i+v^{i+1}[\![k>i]\!])x^2\\&+x(\textit{scw}_k(x,v\mid i-1)+\textit{scw}_k(x,v\mid i)+\textit{scw}_k(x,v\mid i+1),
\end{align*}
where $\textit{scw}_k(x,v\mid j)=0$ for $j\not\in [k]$ and $[\![P]\!]=1$ if the condition $P$ holds,
and $[\![P]\!]=0$ otherwise.
\end{lemma}
\begin{proof}
Let $\sigma$ be any smooth cyclic word containing at least two letters.
If $\sigma$ contains exactly two letters and  $\sigma_1=i$, then $j=i-1,i,i+1$, which gives the contribution $(v^{i-1}[\![i>1]\!]+v^i+v^{i+1}[\![k>i]\!])x^2$. Otherwise, the second letter of $\sigma=i\sigma_2\cdots\sigma_{n-1}j$ is either $i-1,i$ or $i+1$. Hence, in terms of generating functions, we have
\begin{multline}
\textit{scw}_k(x,v\mid i)
=(v^{i-1}[\![i>1]\!]+v^i+v^{i+1}[\![k>i]\!])x^2\\+\bigl(\textit{scw}_k(x,v\mid i-1)+\textit{scw}_k(x,v\mid i)+\textit{scw}_k(x,v\mid i+1)\bigr)x,
\end{multline}
for all $1\leq i\leq k$, which completes the proof.
\end{proof}

Restating Lemma~\ref{lemscw1} as a matrix system we have
\begin{equation}\label{sys2}
\textbf{A}\begin{pmatrix}\textit{scw}_k(x,v\mid 1)\\\textit{scw}_k(x,v\mid 2)\\\vdots\\ \textit{scw}_k(x,v\mid {k-1})\\\textit{scw}_k(x,v\mid k)\end{pmatrix} =\begin{pmatrix}v+v^2\\v+v^2+v^3\\\vdots\\ v^{k-2}+v^{k-1}+v^k\\v^{k-1}+v^k\end{pmatrix}x^2,
\end{equation}
where $\textbf{A}$ is the tridiagonal matrix already defined in the previous section.

\begin{theorem}\label{th2}
The generating function for the number of smooth cyclic words of length $n$ over an alphabet of $k$ letters
is given by
$$\textit{scw}_k(x)=1+\frac{kx(1+3x)}{(1+x)(1-3x)}-\frac{2(k+1)x}{(1+x)(1-3x)}\frac{U_{k-1}\left(\frac{1-x}{2x}\right)}{U_k\left(\frac{1-x}{2x}\right)}.$$
\end{theorem}
\begin{proof}
Equation~\eqref{sys2} gives that
$$\begin{pmatrix}\textit{scw}_k(x,v\mid 1)\\\textit{scw}_k(x,v\mid 2)\\\vdots\\ \textit{scw}_k(x,v\mid {k-1})\\\textit{scw}_k(x,v\mid k)\end{pmatrix} =\textbf{A}^{-1}\begin{pmatrix}v+v^2\\v+v^2+v^3\\\vdots\\ v^{k-2}+v^{k-1}+v^k\\v^{k-1}+v^k\end{pmatrix}x^2,$$
where $\textbf{A}^{-1}$ is defined in \eqref{invA}.

Fix $t=\frac{1-x}{2x}$ and $i$, where $i=1,2,\ldots,k$. By comparing the coefficients of $v^j$ in the $i$-th row in the above matrix equation we obtain, for $i\neq 1,k$,
$$\sum_{j=1}^k \textit{scw}_k(x\mid i\mid j)=x^2(\textbf{A}^{-1}_{i(i-2)}+2\textbf{A}^{-1}_{i(i-1)}+3\textbf{A}^{-1}_{ii}
+2\textbf{A}^{-1}_{i(i+1)}+\textbf{A}^{-1}_{i(i+2)}),$$
and for $i=1,k$  we have
\begin{align*}
\sum\limits_{j=1}^2\textit{scw}_k(x\mid 1\mid j)&=x^2(2\textbf{A}^{-1}_{11}+2\textbf{A}^{-1}_{12}+\textbf{A}^{-1}_{13}),\\
\sum\limits_{j=k-1}^k\textit{scw}_k(x\mid k\mid j)&=x^2(2\textbf{A}^{-1}_{kk}+2\textbf{A}^{-1}_{k(k-1)}+\textbf{A}^{-1}_{k(k-2)}).
\end{align*}
Note that $\textit{scw}_k(x\mid i\mid j)=0$ for $|j-i|>1$. Thus the generating function $\textit{scw}_k(x)$ is given by
$$1+kx-x^2(\textbf{A}^{-1}_{11}+\textbf{A}^{-1}_{kk})+x^2\sum_{i=1}^k(\textbf{A}^{-1}_{i(i-2)}+2\textbf{A}^{-1}_{i(i-1)}+3\textbf{A}^{-1}_{ii}+2\textbf{A}^{-1}_{i(i+1)}+\textbf{A}^{-1}_{i(i+2)}),$$
where $1$ counts the empty words, and $kx$ counts the words of length $1$ in the set of words over $[k]$.
Therefore, applying \eqref{invA} we obtain
\begin{align*}
\textit{scw}_k(x)&=1+kx+\dfrac{2x(2U_{k-1}(t)+2U_{k-2}(t)+U_{k-3}(t))}{U_k(t)}\\
&+\dfrac{x}{U_k(t)}\sum\limits_{i=2}^{k-1}\bigl(U_{i-3}(t)+2U_{i-2}(t)+3U_{i-1}(t)\bigr)U_{k-i}(t)\\&+U_{i-1}(t)(2U_{k-1-i}(t)+U_{k-2-i}(t)),
\end{align*}
which, by simple algebraic operations,  is equivalent to
\begin{align*}
\textit{scw}_k(x)&=1+kx+\dfrac{2x(2U_{k-1}(t)+2U_{k-2}(t)+U_{k-3}(t))}{U_k(t)}\\
&+\dfrac{x}{U_k(t)}\sum\limits_{i=0}^{k-3}\bigl(2U_{i-1}(t)+4U_{i}(t)+3U_{i+1}(t)\bigr)U_{k-2-i}(t).
\end{align*}
and can be simplified to%
\footnote{Such identities can be proved using a computer, as explained before. It is, however, not so easy to \emph{find} them.}
\begin{align*}
\textit{scw}_k(x)&=1+kx\\
&+\dfrac{x}{U_k(t)}\biggl[4U_{k-1}(t)+4U_{k-2}(t)+2U_{k-3}(t)\\
&\qquad+\frac{x^2}{(1+x)(1-3x)}\Bigl(3(k-2)U_{k+1}(t)+4(k-2)U_k(t)-kU_{k-1}(t)\\
&\qquad\qquad-4(k-1)U_{k-2}(t)-2(k+1)U_{k-3}(t)-4U_{k-4}(t)-2U_{k-5}(t)\Bigr)\biggr].
\end{align*}
Using the recursion \eqref{eqcheb} for the Chebyshev polynomials several times we arrive at
$$\textit{scw}_k(x)=1+\frac{kx(1+3x)}{(1+x)(1-3x)}-\frac{2(k+1)x}{(1+x)(1-3x)}\frac{U_{k-1}(t)}{U_k(t)},$$
as claimed.
\end{proof}

Now let us find an explicit formula for the number of cyclic smooth words of length $n$ over the alphabet $[k]$.

\begin{theorem}\label{th22}
The number of smooth cyclic words of length $n$ over the alphabet $[k]$ is given by
$$\textit{scw}_{n,k}=\sum\limits_{j=1}^k\left[1+2\cos\Bigl(\frac{j\pi}{k+1}\Bigr)\right]^n.$$
\end{theorem}
\begin{proof}
Fix $k$ and $\theta_j=\frac{j\pi}{k+1}$. Then  Lemma~\ref{lem1u} implies that the coefficient of $x^n$ in $\frac{2(k+1)x}{(1+x)(1-3x)}\frac{U_{k-1}(t)}{U_k(t)}$, with $t=\frac{1-x}{2x}$, is given by
$$q_{n,k}=[x^n]\frac{4x^2}{(1+x)(1-3x)}\sum_{j=1}^k\frac{\sin^2\theta_j}{1-x(1+2\cos\theta_j)}.$$
Using the fact that
$$\frac{1}{(1+x)(1-3x)(1-x\omega)}=\frac{1}{4(1+\omega)(1+x)}+\frac{\omega^2}{(\omega-3)(1+\omega)(1-x\omega)}-
\frac{9}{4(\omega-3)(1-3x)},$$
and \eqref{eqc22}, we obtain that
\begin{align*}
q_{n,k}&=\sum\limits_{j=1}^k\sin^2\theta_j\left[\frac{(-1)^n}{4\cos^2(\frac{\theta_j}{2})}-\frac{(1+2\cos\theta_j)^n}{\sin^2\theta_j}
+\frac{3^n}{4\sin^2(\frac{\theta_j}{2})}\right]\\
&=\sum\limits_{j=1}^k\left[\sin^2(\frac{\theta_j}{2})(-1)^n-(1+2\cos\theta_j)^n+\cos^2(\frac{\theta_j}{2})3^n\right].
\end{align*}
Using the identities $\sum_{j=1}^k\cos^2(\frac{\theta_j}{2})=\frac{k}{2}$ and  $\sum_{j=1}^k\sin^2(\frac{\theta_j}{2})=\frac{k}{2}$, we get that
\begin{align*}
q_{n,k}&=\frac{k}{2}\bigl((-1)^n+3^n\bigr)-\sum\limits_{j=1}^k(1+2\cos\theta_j)^n.
\end{align*}
Hence, Theorem~\ref{th2} states that the coefficient of $x^n$, $n\geq1$, in the generating function $\textit{scw}_k(x)$ is given by
$$\textit{scw}_{n,k}=\frac{k}{2}\bigl(3^n+(-1)^n\bigr)-q_{n,k}=\sum\limits_{j=1}^k(1+2\cos\theta_j)^n,$$
as claimed.
\end{proof}
\begin{corollary}
Asymptotically, we have as $n \to \infty$,
\begin{equation*}
\textit{scw}_{n,k}\sim\left[1+2\cos\Bigl(\frac{\pi}{k+1}\Bigr)\right]^n.
\end{equation*}
\end{corollary}

Note that, asymptotically, smooth and cyclic smooth words have the same exponential growth order, just a different constant. More precisely we may deduce

\begin{corollary}
The proportion of smooth words that are cyclic smooth in $[k]^n$ tends to $\frac{1}{2} (k+1) \left(2 \cos \bigl(\frac{\pi }{k+1}\bigr)+1\right)
   \tan ^2\bigl(\frac{\pi }{2( k+1)}\bigr)$ as $n \to \infty$.
\end{corollary}

We observe that for large $k$,
\[\frac{1}{2} (k+1) \left(2 \cos \Bigl(\frac{\pi }{k+1}\Bigr)+1\right)
   \tan ^2\Bigl(\frac{\pi }{2 (k+1)}\Bigr)=\frac{3 \pi ^2}{8 k}-\frac{3\pi ^2}{8k^2}
   +O(k^{-3}).
\]

\section{Smooth necklaces} \label{smoothneck}

Smooth necklaces were defined in the introduction of the paper.
To count the number $\textit{sn}_{n,k}$ of smooth necklaces of length $n$ over an alphabet of $k$ letters we consider equivalence classes of smooth cyclic words up to rotation.
From Theorem~\ref{th2} we then obtain the following result by a direct application of Theorem~\ref{th2} and \cite[Exercise 7.112(a)]{St}.

\begin{theorem}\label{th3}
Let $n\geq1$. The number $\textit{sn}_{n,k}$ of smooth necklaces of length $n$ over an alphabet of $k$ letters is given by
$$
\textit{sn}_{n,k}=\frac{1}{n}\sum\limits_{i=1}^k\sum_{j\mid n}\phi(j)\left[1+2\cos\Bigl(\frac{i\pi}{k+1}\Bigr)\right]^{n/j}
$$
where $\phi$ is Euler's totient function ($\phi(n)$ is the number of positive integers $\le n$ that are relatively prime to $n$), and we write $j\mid n$ if $j$ divides $n$.
\end{theorem}

We see from this that asymptotically as $n \to \infty$, $\textit{sn}_{n,k}\sim \textit{scw}_{n,k}$.

The following table (Table~\ref{tab2}) is obtained from Theorem~\ref{th3}.
\begin{table}[htp]
\begin{tabular}{c|cccccccccccc}
  $n$        & 0 & 1 & 2 & 3 & 4 & 5 & 6 & 7  & 8  & 9  & 10  & 11  \\\hline\hline
  $\textit{sn}_{n,1}$ & 1 & 1 & 1 & 1 & 1 & 1 & 1  & 1  & 1  & 1   &  1 &  1  \\\hline
  $\textit{sn}_{n,2}$ & 1 & 2 & 3 & 4 & 6 & 8 & 14 & 20 & 36 & 60  & 108& 188 \\\hline
  $\textit{sn}_{n,3}$ & 1 & 3 & 5 & 7 & 12& 19& 39 & 71 & 152& 315 & 685&1479 \\\hline
  $\textit{sn}_{n,4}$ & 1 & 4 & 7 & 10& 18& 30& 65 & 128& 293& 658 & 1544&3622\\\hline
  $\textit{sn}_{n,5}$ & 1 & 5 & 9 & 13& 24& 41& 91 & 185& 435& 1009& 2445&5945\\\hline
  $\textit{sn}_{n,6}$ & 1 & 6 & 11& 16& 30& 52& 117& 242& 577& 1360& 3347&8278\\\hline
  $\textit{sn}_{n,7}$ & 1 & 7 & 13& 19& 36& 63& 143& 299& 719& 1711& 4249&10611\\\hline
\end{tabular}\vspace{7pt}
\caption{Numbers of smooth necklaces $\textit{sn}_{n,k}$.}\label{tab2}
\end{table}

%============================================================================
\bibliographystyle{plain}

\end{document}